\documentclass[12pt]{article}
\usepackage{geometry}                
\usepackage{graphicx}
\usepackage{amssymb}

\usepackage{amscd}
\usepackage{amssymb}
\usepackage{graphics}
\usepackage{amsmath}
\usepackage{enumerate}
\usepackage[english]{babel}
\usepackage{hyperref}
\usepackage[latin1]{inputenc}

\usepackage[usenames,dvipsnames]{color}      
\usepackage{srcltx}  

\usepackage{amsthm, amssymb}
\usepackage{setspace}
\usepackage[all]{xy}
\usepackage{graphicx}
\usepackage{ifpdf}
\usepackage{babel}
\usepackage{verbatim}

\newcommand{\NN}{{\mathbb{N}}}

\newcommand{\gb}{\beta}

\newcommand{\gC}{\Gamma}
\newcommand{\gc}{\gamma}

\newcommand{\gO}{\Omega}

\newcommand{\ga}{\alpha}

\newcommand{\Sym}{\mathrm{Sym}}
\newcommand{\Alt}{\mathrm{Alt}}

\newcommand{\vol}{\mathrm{vol}}

\newcommand{\SL}{\mathrm{SL}}

\newcommand{\Aut}{\mathrm{Aut}}

\newtheorem{prop}{Proposition}[section]
\newtheorem{thm}[prop]{Theorem}
\newtheorem{lem}[prop]{Lemma}
\newtheorem{cor}[prop]{Corollary}

\theoremstyle{definition}

\newtheorem{theorem}{Theorem}

\newtheorem{remark}[theorem]{Remark}

\def\st{\ : \ }
\def\Alt{\operatorname{Alt}}
\def\Sym{\operatorname{Sym}}

\def\cL{\mathcal L}
\def\Alt{\operatorname{Alt}}

\def\NN{\mathbf N}
\def\ZZ{\mathbf Z}

\def\RR{\mathbf R}
\def\Gn{{\Lambda}}

\title{Simple groups without lattices}
\author{Bader, Caprace, Gelander and Mozes}

\date{First draft: May 2010; revised: March 2011}                                           

\begin{document}

\maketitle

\begin{abstract}
We show that the group of almost automorphisms of a $d$-regular tree does not admit lattices. As far as we know, this is the first such example among  (compactly generated) simple  locally compact groups.
\end{abstract}

\section{Introduction}

Let $G$ be a locally compact group. A \textbf{lattice} in $G$ is a discrete subgroup $\Gamma$ such that $G/\Gamma$ carries a finite $G$-invariant measure. Important and well known examples are provided by $\Gamma = \ZZ^n$ in $G= \RR^n$, or $\Gamma = \SL_n(\ZZ)$ in $G = \SL_n(\RR)$. Despite of the basic nature of the latter objects, we emphasise that the existence of a lattice in a given group $G$ should be considered as a very strong condition on that group. It notably requires $G$ to be unimodular, but this condition is however not sufficient for the existence of a lattice. This is well illustrated by nilpotent Lie groups: all of them are indeed unimodular, but many fail to contain any lattice (see Remark II.2.14 in \cite{Raghunathan}). An example due to I.~Kaplansky (see \cite[Example~2.4.7]{Rudin}) shows that a non-compact Abelian (hence unimodular) locally compact group can even fail to contain any infinite discrete subgroup; such a group contains \emph{a fortiori} no lattice.

The question of existence of lattices is especially interesting in the case where the ambient group is topologically simple (and, hence, necessarily unimodular). The fundamental case of Lie groups is well understood:  according to a well-known theorem due to A.~Borel~\cite{Borel}, \cite[XIV.14.1]{Raghunathan}, every non-compact simple Lie group contains a uniform and a non-uniform lattice. More generally, arithmetic groups provide an important source of lattices in semi-simple algebraic groups over any locally compact field. Beyond the linear world, some non-linear simple locally compact groups are also known to possess lattices. A typical example is the group $\Aut(T)^+$ of type-preserving automorphisms of a regular locally finite tree $T$, which is of index two in the full automorphism group $\Aut(T)$. The group $\Aut(T)^+$ is compactly generated and simple, and contains both uniform and non-uniform lattices.

The purpose of this note is to provide an example of a compactly generated simple locally compact group which does not contain any lattice. In order to describe it, we let $d\ge2$ be a fixed integer, $T$ be a (non-rooted) $(d+1)$-regular tree and $G$ the group of \textbf{almost automorphisms} (also sometimes called \textbf{spheromorphisms}) of $T$. An element in $G$ is defined by a triple $(A,B,\varphi)$ where $A,B\subset T$ are finite subtrees with $|\partial A|=|\partial B|$ and $\varphi:T\setminus A\to T\setminus B$ is an isomorphism between the complements, and two such triples define the same element in $G$ if and only if up to enlarging $A,B$ they coincide.

The group $G$ was first introduced by Neretin~\cite{Neretin}; it is known to be abstractly simple~\cite{Kapoudjian}. For each vertex $v \in T$, the stabilizer $\Aut(T)_v$ is a compact open subgroup of $\Aut(T)$ and it is not difficult to see that $G$ commensurates $\Aut(T)_v$. (In fact, the group $G$ can be identified with the  group of all \textbf{abstract commensurators} of $\Aut(T)_v$ or, equivalently, with the group of    \textbf{germs of automorphisms} of $\Aut(T)$, see \cite{Rover} and \cite[Cor.~E]{CDM}. This fact is however not relevant to our present purposes.)

We endow $G$ with the group topology defined by declaring that the conjugates of $\Aut(T)_v$ in $G$ form a sub-basis of identity neighbourhoods. Since $G$ commensurates $\Aut(T)_v$, it follows that the embedding $\Aut(T) \hookrightarrow G$ continuous. In this way, the group $G$ becomes a totally disconnected locally compact group containing $\Aut(T)$ as an open subgroup. In particular elements of $G$ close to the identity can be realised as true automorphisms of $T$. As a locally compact group, the group $G$ is compactly generated; in fact it contains a dense copy of the Higman--Thompson group $V_{d, 2}$, which is finitely generated (see \cite[Th.~6.10]{CDM}).

\medskip
The main result of this note is the following.

\begin{theorem}\label{thmG}
$G$ does not contain any lattice.
\end{theorem}

Fix an edge $e_0$ in $T$.
By $B_n(e_0)$ we denote the open ball of radius $n$ in $T$ around $e_0$. Thus the boundary sphere $\partial B_n(e_0)$ is a set of vertices, consisting of $k_n=2d^n$ elements,
meeting every connected component of the graph $T\setminus B_n(e_0)$.
 Consider the subgroup $O\le G$ consisting of all elements represented by triples
 $(B_n(e_0),B_n(e_0),\varphi)$.
Thus $O$ can be identified with the increasing union
$$
 O=\bigcup_{n\in\NN}O_n \hspace{1cm} \text{where} \hspace{1cm} O_n=\text{Aut}(T\setminus B_n(e_0)).
$$

The groups $O_n$ are compact and open in $G$ and $O$ is their union. Essential to our argument is that $O$ is open in $G$. Therefore, if $L$ is a (uniform) lattice in $G$ then its intersection with $O$ is a (uniform) lattice in $O$. Thus Theorem~\ref{thmG} is a consequence of the following.

\begin{theorem}\label{thmO}
$O$ does not contain any lattice.
\end{theorem}

Although this fact will not be necessary for our argument, we point out that the group $O$ is itself topologically simple (see \cite[Lem.~6.9]{CDM}), and thus constitutes another example of a simple locally compact group without lattices. However, in contrast to the group $G$, the group $O$ is not compactly generated.

\medskip

We will prove Theorem~\ref{thmO} by way of contradiction. We assume henceforth that $\gC\le O$ is a lattice.
Our argument can be outlined as follows. By construction $O$ is a union of an ascending chain of compact open subgroups $O_n$. Moreover, for $n$ large the profinite group $O_n$ maps onto a full symmetric group $\Sym(k_n)$ of very large degree $k_n$. The image of the intersection $\Gamma \cap O_n$ maps onto a subgroup whose index in $\Sym(k_n)$ is controlled by the covolume of $\Gamma$ in $O$. A precise estimate of that index will be established in the first subsection below. This leads us to studying subgroups of finite symmetric groups of `relatively small index'. In the general case, we shall invoke results due to L.~Babai~\cite{Babai81,Babai82} which are relevant to the latter question in order to complete our study. However, in some special cases, it is possible to complete the proof of Theorem~\ref{thmO} using exclusively elementary methods. This is notably the case if one assumes that $\Gamma$ is cocompact in $O$, or alternatively  if one assumes that $T$ is trivalent (\emph{i.e.} $d=2$). The latter two special cases are presented in separate sections by way of illustration. The reader who is interested in contemplating a single example of a compactly generated simple locally compact group without lattices can read through Sections~2 to 5 below, avoiding the technical complications arising in the general discussion carried in Section~6.

\subsection*{Acknowledgement}
We thank the referee for careful reading of the manuscript and detailed comments which helped in improving the presentation of this  paper.


\section{Some notations and bounds}\label{sec:notations:bound}

We fix an edge $e_0$ in $T$, and denote by $B_n(e_0)$ the open ball of radius $n$ around $e_0$.
We denote its boundary by $K_n=\partial B_n(e_0)$ and view it as a set of vertices in $T$.
Its cardinality is denoted by
$$
k_n=|K_n|=2d^n.
$$
The group $O_0$ as defined above coincides with the stabilizer of $e_0$ in $\Aut(T)$.
We set $U_0=O_0$ and, for every $n\in\NN$ we denote by $U_n\le \Aut(T)$ the pointwise stabilizer of $B_n(e_0)$.
Then $\{ U_n\}$ form a base of identity neighbourhoods for the topology on $G$.
Furthermore, $U_n$ is a normal subgroup of $O_n$ and $O_n/U_n\cong \Sym(K_n)\cong \Sym(k_n)$.
We denote by $\pi_n:O_n\to\Sym(k_n)$ the quotient map, and by
$$
 A_n=\pi(U_0)=U_0/U_n \cong \Aut(B_n(e_0)).$$
The order of the finite group $A_n$ is denoted by
 $$ a_n=|A_n|=2(d!)^{2\cdot\frac{d^n-1}{d-1}}.
$$

Assume that $\gC$ is a lattice in $O$. We let
$$
 \gC_{O_n}=\gC\cap O_n \hspace{1cm} \text{and} \hspace{1cm} \gC_n=\pi_n(\gC_{O_n})\le \Sym(k_n).
$$
Note that since $\gC$ is discrete, there is some $n_0\in\NN$ such that $\gC\cap U_n=\{1\}$ and hence $\gC_n\cong\gC_{O_n}$, for all $n\ge n_0$.
Let $\mu$ be the Haar measure on $O$, normalized by $\mu(U_0)=1$ and let
$$
 c=\vol(O/\gC) \hspace{1cm} \text{and} \hspace{1cm} c_n=\vol(O_n/\gC_{O_n}).
$$
Then the sequence $c_n$ is non-decreasing and tends to
$c$ as $n \to \infty$.
\medskip

For $n\ge n_0$ we have the following volume computation:
\begin{equation}\label{Eq:c_n}
\begin{array}{rcl}
 c_n & = & \displaystyle \vol(O_n/\gC_{O_n})=\frac{\mu (O_n)}{|\gC_{O_n}|}=\frac{[O_n:O_0]}{|\gC_n|}=
 \frac{[O_n:U_n]}{[U_0:U_n]|\gC_n|}=\frac{1}{a_n}\frac{|\Sym(k_n)|}{|\gC_n|}\\
 & =& \displaystyle \frac{[\Sym(k_n):\gC_n]}{a_n}.
 \end{array}
\end{equation}
In particular
\begin{equation}\label{Eq:c}
 [\Sym(k_n):\gC_n]\le c\cdot a_n.
\end{equation}

The latter inequality is the crucial estimate that will be confronted with the discreteness of $\Gamma$ in order to establish a final contradiction. More precisely, in most cases we shall prove that this condition on the index of $\Gamma_n$ will force $\Gamma$ to meet the identity neighbourhood $U_m$ for $m$ arbitrarily large.



\section{The cocompact case}\label{sec:compact}

The purpose of this section is to give a simple proof for the inexistence of \emph{uniform} lattices in $O$.
The proof will make use of the following Lemma.

\begin{lem}\label{lem:2-transitive}
A subgroup of the symmetric group $\Sym(k)$, generated by two prime cycles $\ga,\gb$ whose respective supports intersect nontrivially but are not contained in one another, acts doubly transitively on its support.
\end{lem}

\begin{proof}
Applying a power of either $\ga$ or $\gb$ followed by a power the other one we can map any pair of points in
$\text{Supp}(\ga)\cup\text{Supp}(\gb)$ to a pair $\{x_1, x_2\}$ satisfying
$$
 x_1\in\text{Supp}(\ga)\setminus\text{Supp}(\gb) \hspace{1cm} \text{and} \hspace{1cm}
 x_2\in\text{Supp}(\gb)\setminus\text{Supp}(\ga).
$$
Observing that  the group $\langle\ga,\gb\rangle$ acts transitively on
$$
 \big(\text{Supp}(\ga)\setminus\text{Supp}(\gb)\big)\times\big(\text{Supp}(\gb)\setminus\text{Supp}(\ga)\big),
$$
the desired conclusion follows.
\end{proof}

We now come back to the setting of Theorem~\ref{thmO} and suppose  that $\gC\le O$ is a uniform lattice. By fixing a relatively compact fundamental domain $\gO$ for $\gC$ in $O$, and recalling that $O=\bigcup_n O_n$ and the $O_n$ are compact, open and ascend to $O$, one sees that $\gO\subset O_n$ for all large $n$ and hence the sequence $c_n=\vol(O_n/\gC_{O_n})$ is eventually constant and equal to $c$. In particular $c$ is rational. By Equation (\ref{Eq:c_n}), there is some $n_1$ such that
$$
[{\Sym}(k_n):\gC_n]=c\cdot a_n=c\cdot 2(d!)^{2\cdot\frac{d^n-1}{d-1}},
$$
for all $n>n_1$.
Therefore, for any prime $p\le k_n$ which does not divide the right hand side, the group $\gC_n$ must contain some $p$-Sylow subgroup of $\Sym(k_n)$. Notice that if $p \geq k_n/2$, such a $p$-Sylow is cyclic and generated by a single $p$-cycle (the equality case $p=k_n/2$ is excluded since $k_n/2 = d^n$ is not prime). From the Prime Number Theorem, it follows that for a sufficiently large integer $k$, the interval $[k/2, k]$ contains at least three distinct primes. Therefore, there is some $n>\max\{ n_0,n_1\}$ such that the interval $[k_n/2,k_n]$ contains two primes, say $p,q$, such that $p+3 \leq q$. Conjugating one $p$-cycle in $\gC_n$ by some $q$-cycle one produces two $p$-cycles satisfying the condition of Lemma~\ref{lem:2-transitive}. We can further ensure that the union of the supports of these two $p$-cycles is a set of cardinality~$k \geq p+3$.
We now invoke a theorem of Jordan \cite{Jordan} (see also \cite[Theorem 13.9]{Wielandt}) ensuring that \emph{a primitive group of degree $k$ containing a $p$-cycle, with $p+3\leq k$, is either the full symmetric or the alternating group}.
It follows  that $\gC_n$ contains the alternating group on some subset $X_n$ of size $k > k_n/2+2$.

By the pigeonhole principle, the set $X_n$ contains  two pairs of vertices $x_i,y_i,~i=1,2$ such that $d_T(x_i,y_i)=2$, \emph{i.e.} the vertices $x_i$ and $y_i$ admit a common ``father'' in $K_{n-1}$. The permutation
$$
 \overline\gc=(x_1,y_1)(x_2,y_2)
$$
is an element of $\text{Alt}(X_n)$ and hence belongs to $\gC_n$. However its pre-image $\gc\in\gC_{O_n}$ acts trivially on $K_{n-1}$ and is thus contained in $U_{n-1}$. Since $n>n_0$ we get a contradiction.
\qed


\section{The proof of Theorem~\ref{thmO}}

In this section we will prove Theorem~\ref{thmO}, relying on the following finite-group-theoretic proposition,
which will be proven in the next sections.

\begin{prop} \label{prop:index}
For all $c,d>0$, and $0<\alpha<1$, there exists an integer $n_1$ (depending on $c,d$ and $\alpha$) such that
for every finite set $K$ with $|K|\geq n_1$, every subgroup $\Lambda<\Sym(K)$ satisfying
the index bound
$$[\Sym(K):\Lambda] \leq c\cdot d^{|K|}$$
enjoys the following (non-exclusive) alternative. Either:
\begin{enumerate}[(1)]

  \item there exists a subset $Z\subset K$ with $|Z|>\frac{|K|}{d} + 2$ and $\Alt(Z)< \Lambda$; or

  \item there exist $d$ disjoint subsets $Z_1,Z_2,\ldots,Z_d\subset K$ with
  $$|\bigcup_{i=1}^d Z_i| > (1-\alpha)|K|, \quad \mbox{and} \quad \prod_{i=1}^d \Alt(Z_i)< \Lambda. $$
\end{enumerate}
\end{prop}

\begin{remark}
For the proof given below for Theorem~\ref{thmO} we will need Proposition~\ref{prop:index} for some $\alpha$ satisfying
$\alpha<1/d^2$ (in fact, a careful inspection of the proof below will reveal that it is enough to assume $\alpha<\frac{d-1}{d^2}$, 
but we choose not to obscure the proof with unnecessary detailed arguments).
This is important to note, as we will give in the next section an independent proof of Proposition~\ref{prop:index} for $d=2$ and $\alpha=0.24$.
\end{remark}

\begin{proof}[Proof of Theorem~\ref{thmO}]

We assume by contradiction that $\Gamma$ is a lattice in $O$.
We use the notations and bounds given in Section~\ref{sec:notations:bound}.
By Equation~(\ref{Eq:c}), we have
$$
[\Sym(k_n):\Gamma_n]\le c\cdot a_n = c\cdot 2(d!)^{2\cdot\frac{d^n-1}{d-1}} \leq c'\cdot d^{2d^n}=c'\cdot d^{k_n}
$$
for some appropriate constant $c'$ and every $n\ge n_0$.
We fix $\alpha < 1/{d^2}$.
We apply Proposition~\ref{prop:index} to the set $K=K_n$
and the group $\Lambda=\Gamma_n$ with the constants $c',d$ and $\alpha$ as above.

Fix $n\ge\max\{n_0+2,n_1\}$ (the constant $n_0$ was defined in Section~\ref{sec:notations:bound} and $n_1$ is given by Proposition~\ref{prop:index}).
Then $\Gamma\cap U_{n-2}=\{1\}$
and we infer that $\Gamma_n$ satisfies one of the alternatives (1) or (2) in Proposition~\ref{prop:index}.

Assume first that $\Gamma_n$ satisfies (1).
Then either there is a vertex $u\in K_{n-1}$ with three neighbours $x_1,x_2,x_3\in Z$,
in which  case the $3$-cycle  $(x_1,x_2,x_3)$ belongs to $\Alt(Z)<\Gamma_n$,
or there are two vertices $u,v\in K_{n-1}$ with $x_1,x_2\in Z$ neighbours of $u$ and $y_1,y_2\in Z$ neighbours of $v$.
In the latter case we have  $(x_1,x_2)(y_1,y_2) \in \Alt(Z)<\Gamma_n$.
Observe that the preimages in $\Gamma_{O_n}$ of both of these elements actually belong to $U_{n-1}$, which
gives as a contradiction, as $U_{n-1}<U_{n-2}$ and $\Gamma\cap U_{n-2}=\{1\}$.

Next suppose that $\Gamma_n$ satisfies the alternative (2).
As in the previous case, if we had for some $1\leq i\leq d$, 
two vertices $u,v\in K_{n-1}$ with $x_1,x_2\in Z_i$ neighbours of $u$ and $y_1,y_2\in Z_i$ neighbours of $v$,
we would get a contradiction.
So, for each $i$ we have at most one element in $K_{n-1}$ with two neighbours in $Z_i$.
Call these {\bf flexible} elements of $K_{n-1}$. There are at most $d$ such.
Call an element of $K_{n-2}$ {\bf flexible} if it has a flexible neighbour in $K_{n-1}$.
Clearly, there are at most $d$ flexible elements in $K_{n-2}$ too.

Say that a vertex $u \in B_k(e_0)$ with $k<n$  is \textbf{fully-covered} 
if the following conditions holds: for every vertex $v\in K_n=\partial B_n(e_0)$ so that $u$ belongs to the geodesic connecting $v$ and $e_0$,
we have $v \in \bigcup_{i=1}^d Z_i$.
Since $|K_n-\bigcup_{i=1}^d Z_i| < \alpha k_n$, the number of  vertices  which are not fully covered in each level $K_j$ is strictly smaller than $\alpha k_n$.
Consider the level $K_{n-2}$.
Since $k_{n-2}=k_n/d^2$ and $\alpha<1/d^2$, we get that there are at least $(1/d^2-\alpha)k_n$ fully covered vertices in $K_{n-2}$.
Picking $n$ large enough so $(1/d^2-\alpha)k_n \geq d+2$ we get at least two vertices in $K_{n-2}$ which are fully covered but not flexible.

For a given vertex $v\in K_{n-2}$ which is fully covered and not flexible we can construct an element of $\Aut(B_n(e_0))$ 
which, as a permutation of $K_n$, belongs to  $\prod_{i=1}^d \Sym(Z_i)$, by picking two neighbours of $v$ in $K_{n-1}$, switching them,
and switching all their neighbors in $K_n$, preserving the sets $Z_i$.
Observe that this element is trivial on $K_{n-2}$.
Having two vertices in $K_{n-2}$ which are fully covered and not flexible, we can compose two such automorphisms, and get an element
of $\prod_{i=1}^d \Alt(Z_i)<\Lambda = \Gamma_n$.
It follows that such an element can be realized as an element $\gamma\in \Gamma$.
Since by construction, $\gamma\in U_{n-2}$, 
we get that $\Gamma\cap U_{n-2} \neq \{1\}$. This is a contradiction.
\end{proof}

\section{The case $d=2$}

In this section we prove Proposition~\ref{prop:index} in the special case $d=2$ and $\alpha=0.24$,
which in this case reads:

\begin{prop} \label{prop:index,d=2}
For each $c>0$, there exists an integer $n_1$ (depending on $c$) such that
for every finite set $K$ with $|K|\geq n_1$, every subgroup $\Lambda<\Sym(K)$ with
$$[\Sym(K):\Lambda] \leq c\cdot 2^{|K|}$$
enjoys the following (non-exclusive) alternative. Either
\begin{enumerate}[(1)]

  \item there exists a subset $Z\subset K$ with $|Z|>\frac{|K|}{2} + 2$ and $\Alt(Z)< \Lambda$, or

  \item there exists two disjoint subsets $Z_1,Z_2 \subset K$ with
  $$|Z_1 \cup Z_2| > 0.76|K|, \quad \mbox{and} \quad \Alt(Z_1)\times \Alt(Z_2)< \Lambda. $$
\end{enumerate}
\end{prop}

Our proof relies on the Prime Number Theorem which can be formulated as
$$
\lim_{n \to \infty}  \big({\prod_{p~\text{prime}~<n}p}\big)/{e^n} = 1.
$$
\begin{lem}\label{clm:p,q}
Fix $c>0$, and let $\Lambda<\Sym(k)$ be a subgroup with $[\Sym(k):\Lambda]\leq c\cdot 2^k$.
For all sufficiently large $k$, there are two primes $p,q\in [0.3k,k]$ such that $\Lambda$ contains a copy of the $p$-Sylow as well as of the $q$-Sylow subgroup of $\Sym(k)$. Furthermore we may assume that $q \geq p+3$ and that $q\ne k/2+1$.
\end{lem}

\begin{proof}
The Prime Number Theorem ensures that the product of all primes smaller than $k$ is approximately $e^k$. For a prime $p$ denote by ${i_p}$ the multiplicity of $p$ in $k!$. Our claim is that for some $p,q$ as above
$|\Lambda|$ is divisible by $p^{i_p}q^{i_q}$. Indeed, if this is not the case, then the index of $\Lambda$ in $\Sym(k_n)$ is divisible by the product of all, except perhaps three or less, primes in the interval $[0.3k,k]$ which is roughly $e^{0.7k}$. However $e^{0.7}>2$ contradicting the fact that $[\Sym(k_n):\Lambda]<c\cdot 2^{k}$.
\end{proof}

We shall make use of the following consequence of Lemma~\ref{lem:2-transitive}.

\begin{cor}\label{cor:2trans}
Let $\ga_1,\ldots,\ga_t$ be prime cycles in $\Sym(k)$ such that for every $1<i\le t$ there is some $1\le j<i$ such that $\ga_i,\ga_j$ satisfy the condition on $\ga,\gb$ in Lemma~\ref{lem:2-transitive}. Then the group $\langle \ga_1,\ldots, \ga_t \rangle$ is doubly transitive on its support.
\end{cor}

\begin{proof}
Given two subgroups $A, B \leq \Sym(k)$ which are doubly transitive on their respective support, if these supports intersect in a subset of cardinality at least two, then it follows that $\langle A \cup B\rangle$ is doubly transitive on its own support. In view of this observation, the desired statement follows from Lemma~\ref{lem:2-transitive} by induction on $t$.
\end{proof}

\begin{proof}[Proof of Proposition \ref{prop:index,d=2}]
Let $K$ be a set such that $|K|=k$ and $[\Sym(K):\Lambda]<c\cdot 2^k$.
Suppose that $k\ge 100$ and is large enough so that the conclusion of Lemma~\ref{clm:p,q} holds, and let
$p,q>0.3k$ be two primes such that $q \geq p+3$ and $\Lambda$ contains a $p$-Sylow and a $q$-Sylow subgroups of $\Sym(K)$, as ensured by Lemma~\ref{clm:p,q}. Note that every $p$-Sylow subgroup of $\Sym(K)$ contains $\lfloor \frac{k}{p} \rfloor$ cycles of length $p$ with disjoint supports, and the same applies to $q$.  Let $c_p^1,\ldots,c_p^r,~r= \lfloor \frac{k}{p} \rfloor\le 3$ (resp. $c_q^1,\ldots,c_q^s,~s= \lfloor \frac{k}{p} \rfloor$) be disjoint $p$-cycles (resp. $q$-cycles) of $\Lambda$.

Fix $j \in \{1, \dots, s\}$. We claim that there is a subgroup $\Lambda_j \leq \Lambda$ generated by $p$-cycles, which is $2$-transitive on its support $K_j = \mathrm{Supp}(\Lambda_j)$ and such that $K_j$ contains $\mathrm{Supp}(c^j_q)$.
Since the complement of the union of the supports of the $c_p^i$'s is of size strictly smaller than $p<q$, the support of $c_q^j$ overlaps with the support of some $c_p^i$. By conjugating $c_p^i$ by a suitable power of $c_q^j$ we obtain a partner with which $c_p^i$ satisfies the condition of Lemma~\ref{lem:2-transitive}. In fact, denoting by $\alpha_1, \dots, \alpha_q$ the various conjugates of $c_p^i$ under the elements of $\langle c_q^j \rangle$ and upon reordering the $\alpha_k$'s appropriately, we obtain a sequence of $p$-cycles which satisfies the hypotheses of Corollary~\ref{cor:2trans} and such that the union of the supports of the $\alpha_k$'s contains $\mathrm{Supp}(c^j_q)$. This proves the claim.

In view of Corollary~\ref{cor:2trans}, we may further assume, upon enlarging  $\Lambda_j$ if necessary, that every $p$-cycle in $\Lambda$ is either contained in $\Lambda_j$ or has support disjoint from $K_j$.

Applying the aforementioned result of C.~Jordan (see section~\ref{sec:compact}) to $\Lambda_j$ and recalling that $|K_j| \geq q \geq p+3$, we deduce that $\Lambda_j$ contains the alternating group on its support. In fact, since $\Lambda_j$ is generated by odd cycles, we have $\Lambda_j = \Alt(K_j)$.

From the property that every $p$-cycle in $\Lambda$ is either contained in $\Lambda_j$ or has support disjoint from $K_j$, it follows that for all $j, j' \in \{1, \dots, s\}$, either $K_j = K_{j'}$ or $K_j\cap K_{j'} =\varnothing$.

If now some $K_j$ has cardinality at least $k/2+2$, then we are done showing the first alternative in Proposition~\ref{prop:index,d=2}
(Notice that this happens for instance in case $s=1$, and for this particular case we made the assumption $q\ne\frac{k}{2}+1$).
Otherwise we have $s >1$ and the sets $K_j$ are pairwise disjoint, and each of them has cardinality at most $k/2+1$.
Again, the property that  every $p$-cycle in $\Lambda$ is either contained in $\Lambda_j$ or  has support disjoint from $K_j$ implies that the sets $K_1, K_2$ and $K \setminus (K_1 \cup K_2)$ constitute blocks of imprimitivity for the $\Lambda$-action. It follows that $\Lambda_j$ admits a subgroup of index $\le 6$ which preserves each of the blocks. Denoting the sizes of these blocks by $a k, b k$ and $r k$ respectively with $a+b+r=1$, and bearing in mind that $a,b\le 0.51$ and $r\le 0.4$ one immediately derives that $r\le 0.24$ since otherwise $|\Lambda|$ would be bounded above by $6(ak)!(bk)!(rk)!\le 0.49^{k} k!$ contradicting the fact that the index $[\Sym(K):\Lambda]\le c\cdot 2^{k}$. Thus, the second alternative in Proposition~\ref{prop:index,d=2} holds.
\end{proof}


\section{The proof of Proposition~\ref{prop:index}}

Our goal in this section is to prove Proposition~\ref{prop:index} for an arbitrary $d$ (independently of   the previous subsection).
Consider a set $K$ of size $k$ and a subgroup $\Gn< \Sym(K)$ such that  $[\Sym(K):\Gn] \in O(d^k)$.
Given $\varepsilon>0$, an orbit will be called \textbf{$\varepsilon$-large} if it is of size at least $\varepsilon|K|$.

\begin{lem}\label{large:orbits}
For each $\delta \in (0, 1]$, there is some $\varepsilon=f_1(d,\delta)>0$ such that whenever $k=|K|$ is sufficiently large, the $\varepsilon$-large orbits cover  at least a proportion of $(1-\delta)$ of the set $K$.
\end{lem}

\begin{proof}
We claim that taking $\varepsilon=f_1(d,\delta)= \frac{\delta}{100(d+1)^{1/\delta}}$ is sufficient.
Suppose that the assertion does not hold. Then there is a subset $Z=\bigcup_{i=1}^t Z_i\subset K$ of size $|Z|> \delta |K|$ which is covered by orbits  $\{Z_i\}_{i=1}^t$ with $|Z_i| < \varepsilon|K|$ for each $i$. This implies that the restriction of $\Gn$ to $Z$ is of index at least
$$\left(
 \begin{array}{c}
 |Z| \\ |Z_1|, |Z_2|,\dots,|Z_t| \end{array}\right).$$
 This multinomial coefficient is thus also a lower bound on the index $[\Sym(K) : \Lambda]$. 
 Let us denote $z=|Z|$ and $z_i=|Z_i|$. Thus we have for each $1\le i \le t$ that $\frac{z}{z_i}>100(d+1)^{1/\delta}$ and hence that $\frac{z}{z_i+1}>50(d+1)^{1/\delta}$.
 Recall the estimate:
 $$
 \left(\frac{m}{e}\right)^m \le m !  \le \left(\frac{m+1}{e}\right)^{m+1} (m+1).
 $$
Hence we have:
\begin{eqnarray*}
&&
 \left(
 \begin{array}{c}
z \\ z_1,z_2,\dots,z_t \end{array}\right) \ge
\frac{\left(\frac{z}{e}\right)^z}{\prod_{i=1}^t\left(\left(\frac{z_i + 1}{e}\right)^{z_i}( z_i + 1)\right)}=
\\ &&
= \frac{z^z}{\prod_{i=1}^t\left( (z_i + 1)^{z_i}( z_i + 1)\right)}
= \prod_{i=1}^t\left( (\frac{z}{ z_i+1})^{z_i}(z_i+1)^{-1}\right) \ge
\\&&
\ge \prod_{i=1}^t\left( (50(d+1)^{1/\delta})^{z_i}(z_i+1)^{-1}\right) = (50(d+1)^{1/\delta})^{z}\prod_{i=1}^t (z_i+1)^{-1} \ge
\\&&
\ge (50(d+1)^{1/\delta})^{z}\prod_{i=1}^t (2z_i)^{-1} \ge \frac{(50(d+1)^{1/\delta})^{z}}{ (2z/t)^t},
\end{eqnarray*}
where the last inequality follows from the inequality $\frac{ a_1 + \dots + a_t } t \geq \sqrt[t]{a_1 \dots a_t}$ between the arithmetic and geometric means. 
 The denominator of the last term is maximal when $t=\frac{2z}{e}$, hence we deduce:
 \begin{eqnarray*}
 &&
 \frac{(50(d+1)^{1/\delta})^{z}}{ (2z/t)^t}  \ge  \frac{(50(d+1)^{1/\delta})^{z}}{ e^{2z/e}}\ge
  (10(d+1)^{1/\delta})^{z}\ge
 \\&&
 \ge (10(d+1)^{1/\delta})^{\delta|K|}\ge (d+1)^k.
 \end{eqnarray*}

 This lower bound however is too large compared to the bound $[\Sym(K):\Gn] \in O(d^k)$.
\end{proof}

\begin{lem}\label{limited:blocks}
There is some $f_2(d,\varepsilon)$ such that if $k=|K|$ is sufficiently large, for any $\Lambda$-orbit $Y\subset K$ whose size is at least $\varepsilon |K|$, any non-trivial  $\Lambda$-invariant block decomposition of $Y$ contains at most $f_2(d,\varepsilon)$ blocks.
\end{lem}
\begin{proof}
A group acting transitively on a set $Y$ and whose action preserves a block decomposition with $b$ blocks is of index at least $|Y|!/(b! ((|Y|/b)!)^b)$ in $\Sym(Y)$. {This quantity is thus also  a lower bound on $[\Sym(K) : \Lambda]$. We remark that since the block decomposition is non-trivial, we have $b \in (1, \frac y 2]$, where $y = |Y| \geq \varepsilon k$. }

We shall now estimate the function $f(b)= \frac{y!}{(b!((y/b)!)^b}$ in the range $b\in [1,\frac{y}{2}]$.
To this end, we use Stirling's approximation under the following form:
\begin{itemize}
\item
 For every $m\in \NN$ we have $m! \ge \sqrt{2\pi m} \left(\frac{m}{e}\right)^m$.
 \item There is some constant $c>1$ so that for every $m\in \NN$ we have $m! \le c \sqrt{2\pi m} \left(\frac{m}{e}\right)^m$.
 \end{itemize}
Using these we have:
\begin{eqnarray*}
&&
 \frac{y!}{b!((y/b)!)^b} \ge \frac{\sqrt{2\pi y}\left(\frac{y}{e}\right)^y}{c\sqrt{2\pi b}\left(\frac{b}{e}\right)^b\left(c\sqrt{2\pi\frac{y}{b}}\left(\frac{y}{b e}\right)^{y/b}\right)^b} =
 \\ &&
 =
  \frac{\sqrt{2\pi y}\left(\frac{y}{e}\right)^y}{c\sqrt{2\pi b}\left(\frac{b}{e}\right)^b\left(c\sqrt{2\pi\frac{y}{b}}\right)^b\left(\frac{y}{b e}\right)^{y}}  =
 \frac{\sqrt{y/b}}{c}  \cdot \frac{b^y}{((c/e)\sqrt{2\pi y b})^b} >
 \\&&
 >
 \frac{1}{(100c)^y} \cdot \frac{b^y}{\sqrt{y b}^b} =: g(b).
 \end{eqnarray*}

Next we claim that the function $g(b)$ (for any fixed $y$) is \textbf{unimodular} in the range $b\in[1,y/2]$, \emph{i.e.} it increases monotonically till it reaches some maximum and then decreases monotonically. 
In particular, for all $b_0 \in [1, y/2]$ and $b \in [b_0, y/2]$, we have $g(b) \geq \min \{ g(b_0), g(y/2)\}$. 

Now, at the rightmost point $b=y/2$ we have
$$
g(y/2)= \frac{1}{(100c)^y} \cdot \frac{(y/2)^y}{\sqrt{y^2/2}^{(y/2)}}\ge
\left(\frac{\sqrt{y}}{400c}\right)^y.
$$
Hence for  $y\ge (400c(3(d+1))^{1/\varepsilon})^2$ we have $g(y/2)\ge ((3(d+1))^{1/\varepsilon})^y\ge (3(d+1))^k$. (Note that we may assume $y\ge (400c(3(d+1))^{1/\varepsilon})^2$ is satisfied since the right hand side is a constant.)

Consider  next  $b_0= 300c(d+1)^{1/\varepsilon}$. We have
\begin{eqnarray*}
&&
g(b_0)= g( 300c(d+1)^{1/\varepsilon})=
 \frac{1}{(100c)^y} \cdot \frac{(300c(d+1)^{1/\varepsilon})^y}{\sqrt{y\cdot 300c(d+1)^{1/\varepsilon}}^{300c(d+1)^{1/\varepsilon}}}
 =
 \\&&
 =
  \frac{(3(d+1)^{1/\varepsilon})^y}{\sqrt{y\cdot 300c(d+1)^{1/\varepsilon}}^{300c(d+1)^{1/\varepsilon}}}=
   \frac{3^y}{\sqrt{y\cdot 300c(d+1)^{1/\varepsilon}}^{300c(d+1)^{1/\varepsilon}}} (d+1)^{y/\varepsilon}.
 \end{eqnarray*}

Hence assuming, as we may, that $k$ (and hence $y$) is larger than a fixed (computable) constant we have
$g( b_0) \ge (d+1)^k$. Thus, for all $b \in  [b_0, y/2]$, we have $g(b) \geq \min \{ (d+1)^k, (3(d+1))^k\}= (d+1)^k $. This gives a lower bound on the index $[\Sym(K) : \Lambda]$ which is larger than $O(d^k)$. It follows that we must have $b < b_0$. In other words, we may choose the requested constant $f_2(d, \varepsilon)$ to equal $b_0=300c(d+1)^{1/\varepsilon}$.

It remains to show the unimodularity of the function $g(x)$ in the interval $[1,y/2]$. As $y$ is fixed we need to consider the function
$$g_1(x)=\frac{x^y}{\sqrt{yx}^x}.$$
It is more convenient to consider its logarithm:
$$h(x)= \log g_1(x) = y\log x - \frac{x}{2}\log(yx)
=(y-\frac{x}{2})\log x -\frac{x}{2}\log y.
$$
Computing its derivative we have:
$$
h'(x)=\frac{y}{x} -\frac{1}{2} - \frac{1}{2}\log x -\frac{\log y}{2}.
$$
This is a monotonly decreasing function of $x>0$ and hence the function $g(x)$ is unimodular.
\end{proof}

\begin{prop} \label{prop:L}
For all $c,d>0$ and $\delta>0$, there is some $C$ such that the following holds.
For every large enough finite set $K$ and for every subgroup $\Lambda<\Sym(K)$ with
$[\Sym(K):\Lambda] \leq c\cdot d^{|K|}$
there exists a collection $\cL$   of pairwise disjoint subsets of $K$,
satisfying the following properties:
\begin{enumerate}[(1)]
  \item\label{large:cover} $|\bigcup_{Z\in\cL} Z | \geq (1-\delta)|K|$.

  \item $[\Sym(\bigcup_{Z\in\cL} Z):\prod_{Z\in\cL} \Alt(Z)] \leq C\cdot d^{|K|}$.

  \item $\prod_{Z\in \cL} \Alt(Z)< \Lambda$.

\end{enumerate}
Furthermore, there exist $\varepsilon_0$ and $V_0$ (also depending only on $c,d$ and $\delta$) such that
$|\cL|\le V_0$ and for each $Z\in \cL$, $|Z|\ge \varepsilon_0 |K|$.
\end{prop}

The proof of Proposition~\ref{prop:L} relies on the following result of L.~Babai.

\begin{thm}[L.~Babai \cite{Babai81} and \cite{Babai82}] \label{Babai:primitive:size}
 A primitive subgroup  $L<\Sym(n)$ which does not contain the alternating group $\Alt(n)$ satisfies:
\begin{itemize}
  \item $|L|<e^{4\sqrt{n}\log^2 n}$ if $L$ is not $2$-transitive, and
  \item $|L|< e^{e^{c\sqrt{\log n}}}$ if $L$ is $2$-transitive.\qed
\end{itemize}
\end{thm}


We point out that these estimates can be strengthened to $|L| < 50 e^{\sqrt n \log n}$ using the Classification of the Finite Simple Groups, see \cite[Cor.~1.1(ii)]{Maroti}. The bounds given by Babai's theorem (which is independent of the Classification of the Finite Simple Groups) will however be sufficient for our purposes.
Note also that the better estimate  $n^{c(logn)^2}$ for the case of $2$-transitive groups has been obtained in \cite{Pyber} (with a simpler argument).

\begin{proof}[Proof of Proposition~\ref{prop:L}]
Denote $k=|K|$ and assume it is large enough so that Lemma~\ref{large:orbits} and Lemma~\ref{limited:blocks} hold.
Let $\varepsilon=f_1(d,\delta)$ be as in Lemma~\ref{large:orbits}.
Let $Y_j$, $j=1,\dots,m$ be all the $\varepsilon$-large orbits.
Note that $m\leq 1/\varepsilon$.
By Lemma~\ref{large:orbits}, we have $|\bigcup_{j=1}^m Y_j|\geq (1-\delta)k$.
For each $1\le j \le m$ let $\{ Y_{j,i} \st 1\le i \le m_j\}$ be a block decomposition of $Y_j$ with the largest possible number of non-trivial blocks.
Set
$$\cL=\{Y_{j,i}~|~j=1,\ldots,m~;~i=1,\ldots,m_j\}. $$
Property (1) of the proposition is clear.

Set $V=f_2(d,\varepsilon)$ as given by Lemma~\ref{limited:blocks} and $V_0=V/\varepsilon$. Then we have $m_j\le V$ for each $j$ and
$|\cL|\leq mV\leq V_0$.
Also for each $j$ we have $|Y_j|\geq \varepsilon k$ and  $Y_{j,i} = |Y_j| / m_j \geq   \varepsilon k/ V$.
Setting $\varepsilon_0=\varepsilon/V$ we have proven the ``furthermore" part of the proposition.

Observe that the index of $\Lambda$ in $\Sym(K)$ gives an upper bound on the index of the restriction $\bar\Lambda$ of $\Gn$ to $\bigcup_{j=1}^m Y_j$ in $\Sym(\bigcup_{j=1}^m Y_j)$. Thus
$$[\Sym(\bigcup_{j=1}^m Y_j):\bar\Lambda]\le [\Sym(K):\Lambda] \leq c\cdot d^k.$$
Assume property (3) holds.
Then we have
$$[\bar\Lambda:\prod_{j=1}^m \prod_{i=1}^{m_j} \Alt(Y_{j,i})]\le \prod_{j=1}^m(m_j ! 2^{m_j})\le (V!2^{V})^{1/\varepsilon}.$$
Thus property (2) follows for $C=c\cdot(V!2^V)^{1/\varepsilon}$.

We are left to prove property (3).
For each $1\le j \le m$ and $1 \le i \le m_j$, we consider the action
of $\Gn$ on the orbit $Y_j$.
Let $\Gn_{j,i}$ be the subgroup which preserves the block $Y_{j,i}$. Remark that we have the index bound  $[\Gn:\Gn_{j,i}]\le m_j \le V$.
Hence $[\Sym(K):\Gn_{j,i}]$ is bounded by $O(d^k)$. By definition the $Y_{j,i}$'s provide the finest $\Lambda$-invariant block decomposition of $Y_j$. This implies that the action of $\Gn_{j,i}$ on $Y_{j,i}$, $1\le i\le m_j$ is primitive.

Using the fact that each of the sets $Y_{j,i}$ is of size $|Y_{j,i}|>  \varepsilon_0|K|$,  it follows from Babai's estimates from Theorem~\ref{Babai:primitive:size} that the restriction of $\Gn_{j,i}$ to $Y_{j,i}$ must contain the corresponding alternating group $\Alt(Y_{j,i})$, otherwise we would deduce that the index of $\Gn_{j,i}$ in $\Sym(K)$ is not bounded by $O(d^k)$. 

We infer that the intersection of $\Gn$ with   $\prod_{j=1}^m \prod_{i=1}^{m_j} \Alt(Y_{j,i})$ is a subgroup which projects onto each of the non-abelian, simple factors $\Alt(Y_{j,i})$. Such a subgroup is either the full product or is of index which is at least the order of some $\Alt(Y_{j,i})$. The latter possibility is excluded because $|\Alt(Y_{j,i})|$ has a much bigger growth rate than  the index $[\Sym(K):\Gn] \in O(d^k)$.
\end{proof}

Recall the definition of the \textbf{entropy function}: For $\alpha_i\ge 0$, $\sum_{i=1}^s \alpha_i=1$, this is the function
$$
H(\alpha_1,\ldots,\alpha_s)=-\sum_{i=1}^s \alpha_i \log_2 \alpha_i.
$$
For a fixed $s$, the multinomial coefficient satisfies
 $$
 \lim_{n\to\infty}\frac{ \log_2 \left(\begin{array}{c}
n
 \\
 \alpha_1 n,\alpha_2 n,\dots,\alpha_s n\end{array}\right)}
{nH(\alpha_1,\alpha_2,\dots,\alpha_s)} = 1.
$$
Hence given $\beta>0$ for sufficiently large $n$
$$
\left(\begin{array}{c}
n
 \\
 \alpha_1 n,\alpha_2 n,\dots,\alpha_s n\end{array}\right)
 \ge
\left(2^{nH(\alpha_1,\alpha_2,\dots,\alpha_s)}\right)^{1-\beta}
$$

\begin{proof}[Proof of Proposition~\ref{prop:index}]

Let $c,d$ and $\alpha$ be given.
Without loss of generality we may assume $\alpha<\frac{1}{d}$. 
Our goal is to show that
there is a choice of $\delta$ (depending on $c,d$ and $\alpha$) such that for each large enough finite set $K$ and every collection $\cL$ of subsets of $K$ afforded by  Proposition~\ref{prop:L} with that choice of $\delta$,
either $\cL$ contains a set of size greater than $\frac{|K|}{d}+2$
or $\cL$ contains $d$ sets whose union is of size greater than $(1-\alpha)|K|$. The conclusion of Proposition~\ref{prop:index} will thus follow.

Observe that
$$ d= 2^{H(\frac{1}{d},\frac{1}{d},\dots,\frac{1}{d})} < 2^{H(\frac{1}{d},\frac{1}{d},\dots,\frac{1}{d},\frac{1}{d}-\frac{\alpha}{2},\frac{\alpha}{2})}. $$
Fix $\tilde{d}$ such that $d< \tilde{d} < 2^{H(\frac{1}{d},\frac{1}{d},\dots,\frac{1}{d},\frac{1}{d}-\frac{\alpha}{2},\frac{\alpha}{2})}$,
and choose $\delta>0$ and $\beta>0$ small enough such that
$$
\delta< \frac{\alpha}{2} \quad \mbox{and} \quad \left(\frac{1}{d}-\alpha\right)^{\delta}
 \left(2^{H(\frac{1}{d},\frac{1}{d},\dots,\frac{1}{d},\frac{1}{d}-\frac{\alpha}{2},\frac{\alpha}{2})}\right)^{1-\beta}>\tilde{d}.
$$
We claim that this choice of $\delta$ will do the job.
We will prove it by a contradiction.
We will assume that
there is no set in $\cL$ of size greater that $\frac{|K|}{d}+2$
and that the union of the $d$ largest sets in $\cL$ covers at most $(1-\alpha)$ of the set $K$.

Set $k=|K|$. By Property (2) of Proposition \ref{prop:L} we have
$$[\Sym(\bigcup_{Z\in\cL} Z):\prod_{Z\in\cL} \Alt(Z)] \leq C\cdot d^{k}.$$
Denoting $\cL=\{Z_1,\ldots,Z_t\}$ and recalling that $t\le V_0$, this index coincides, up to a constant, with a multinomial coefficient
$$
\left(
\begin{array}{c}
 |\bigcup Z_i| \\ |Z_1|, |Z_2|,\dots,|Z_t| \end{array}\right)
  $$
We will derive the contradiction by estimating this multinomial coefficient from below, showing that it is too big.

Let us denote
 $z_i=|Z_i|$, $1\le i\le t$ and $z=|\bigcup Z_i|$. Assume, as we may, that $z_1 \ge z_2 \ge z_3 \ge \dots\ge z_t$.
Note that by the assumption we have $z_i\le \frac{k}{d}+2$ for all $1\le i\le t$.
Since $\delta<\frac{\alpha}{2}$ Property (1) of Proposition~\ref{prop:L} implies that $\sum_{j\ge d+1} z_j \ge \frac{\alpha k}{2}$.
In a multinomial coefficient if we keep all the terms fixed except for two terms which we change so that their sum is fixed but their difference increases, then the total value decreases.
We claim that
applying this repeatedly we can deduce that our multinomial is bounded below by:
\begin{equation}\label{minimal:multinomial}
\left(\begin{array}{c}
 z \\ a_1,a_2,\dots,a_{d-1},b,\frac{\alpha k}{2} \end{array}\right)
\end{equation}
where for $1\le i \le d-1$,
$\frac{k}{d} \le a_i \le \frac{k}{d}+2$  and $b= z- \sum a_i - \frac{\alpha k}{2}$. (Note that if some of the fractions above are not integers one has to take nearby integers.)
This is done by an inductive process. Suppose we have already shown that:
\begin{equation*}\label{intermediate:multinomial}
\left(\begin{array}{c}
 z \\ z_1,z_2,\dots,z_t \end{array}\right)
 \ge
 \left(\begin{array}{c}
 z \\ a_1,a_2,\dots,a_{i-1},z_i,z_{i+1},\dots,z_s \end{array}\right)
\end{equation*}
and $i\le d-1$. Then we may transfer as much as we can from the last term $z_s$ to the term $z_i$ as long as the latter does not increase beyond $\frac{k}{d}$ and as long as $\sum_{j\ge d+1} z_j \ge \frac{\alpha k}{2}$.
 If $z_s$ becomes $0$, we remove it.
Once the entry $z_i$ is increased to $\frac{k}{d}$ call it $a_i$ and repeat until $i=d$ or
 $\sum_{j\ge d+1} z_j = \frac{\alpha k}{2}$ .

 If we reached $i=d$ (and $\sum_{j\ge d+1} z_j \ge \frac{\alpha k}{2}$)  increase $z_d$ in the same way till it is  equal to
  $z- \sum_{i=1}^{d-1} a_i - \frac{\alpha k}{2}$. Then one may collect all the remaining terms from the $d+1$ place onward together and get one term which will be $\frac{\alpha k}{2}$.
  If along the process we reached $\sum_{j\ge d+1} z_j = \frac{\alpha k}{2}$ with $i\le d-1$ then we can collect all the terms from the $d+1$ place onwards together to form one term which is equal to $\frac{\alpha k}{2}$ and then we can ``transfer mass" from the term $z_d$ to those entries $z_i,\dots,z_{d-1}$ which are still smaller then $\frac{k}{d}$.
Using this process, we eventually reach the form~(\ref{minimal:multinomial}) and we have

\begin{equation*}\label{intermediate:multinomial:second}
\left(\begin{array}{c}
 z \\ z_1,z_2,\dots,z_t \end{array}\right)
 \ge
 \left(\begin{array}{c}
 z \\ a_1,a_2,\dots,a_{d-1},b,\frac{\alpha k}{2} \end{array}\right)
\end{equation*}
as claimed. Observe now that
since for $1\le i \le d-1$ we have $\frac{k}{d} \le a_i \le \frac{k}{d}+2$ it follows that for some fixed polynomial $p(x)$  (say, $p(x)= (x/d+2)^{d-1}$) we have
\begin{equation}
 \left(\begin{array}{c}
 z \\ a_1,a_2,\dots,a_{d-1},b,\frac{\alpha k}{2} \end{array}\right)
\ge\frac{1}{p(k)}
 \left(\begin{array}{c}
 z \\\frac{k}{d},\frac{k}{d},\dots,\frac{k}{d}, (z-(\frac{d-1}{d}+\frac{\alpha}{2})k), \frac{\alpha k}{2} \end{array}\right)
\end{equation}

We claim that the choice of $\delta$ guarantees that the last multinomial coefficient grows as $\tilde{d}^k$ (which allows us to absorb (i.e. ignore) the polynomial factor $\frac{1}{p(k)}$).
 Using that $z> (1-\delta)k$ and $\delta<\frac{\alpha}{2}$ we have for $k$ large enough:

{\allowdisplaybreaks
\begin{eqnarray*}
&&
\left(\begin{array}{c}
 z \\\frac{k}{d},\frac{k}{d},\dots,\frac{k}{d}, (z-(\frac{d-1}{d}+\frac{\alpha}{2})k), \frac{\alpha k}{2} \end{array}\right)
 \ge
 \\&&
 \ge
 \left(\frac{(z-(\frac{d-1}{d}+\frac{\alpha}{2})k)}{k}\right)^{k-z}
 \left(\begin{array}{c}
 k \\\frac{k}{d},\frac{k}{d},\dots,\frac{k}{d}, (\frac{1}{d}-\frac{\alpha}{2})k, \frac{\alpha k}{2} \end{array}\right)
 \ge
 \\&&
  \ge
 \left((1-\delta)-(\frac{d-1}{d}+\frac{\alpha}{2})\right)^{k-z}
 \left(\begin{array}{c}
 k \\\frac{k}{d},\frac{k}{d},\dots,\frac{k}{d}, (\frac{1}{d}-\frac{\alpha}{2})k, \frac{\alpha k}{2} \end{array}\right)
 \ge
 \\&&
  \ge
 \left(\frac{1}{d}-\delta-\frac{\alpha}{2}\right)^{\delta k}
 \left(\begin{array}{c}
 k \\\frac{k}{d},\frac{k}{d},\dots,\frac{k}{d}, (\frac{1}{d}-\frac{\alpha}{2})k, \frac{\alpha k}{2} \end{array}\right)
 \ge
 \\&&
  \ge
 \left(\frac{1}{d}-\alpha\right)^{\delta k}
 \left(\begin{array}{c}
 k \\\frac{k}{d},\frac{k}{d},\dots,\frac{k}{d}, (\frac{1}{d}-\frac{\alpha}{2})k, \frac{\alpha k}{2} \end{array}\right)
 \ge
 \\&&
 \ge
  \left(\left(\frac{1}{d}-\alpha\right)^{\delta}\right)^ k
 \left(2^{H(\frac{1}{d},\frac{1}{d},\dots,\frac{1}{d},\frac{1}{d}-\frac{\alpha}{2},\frac{\alpha}{2})}\right)^{(1-\beta)k} =
 \\&&
 =
  \left(\left(\frac{1}{d}-\alpha\right)^{\delta}
 \left(2^{H(\frac{1}{d},\frac{1}{d},\dots,\frac{1}{d},\frac{1}{d}-\frac{\alpha}{2},\frac{\alpha}{2})}\right)^{1-\beta}\right)^k > \tilde{d}^k.
  \end{eqnarray*}
  }
Since $\tilde{d}>d$ this clearly contradicts Property (2) of Proposition \ref{prop:L}.

Note that we have used  that for sufficiently large $k$ 
$$
\left(\begin{array}{c}
 k \\\frac{k}{d},\frac{k}{d},\dots,\frac{k}{d}, (\frac{1}{d}-\frac{\alpha}{2})k, \frac{\alpha k}{2} \end{array}\right)
 \ge
  \left(2^{H(\frac{1}{d},\frac{1}{d},\dots,\frac{1}{d},\frac{1}{d}-\frac{\alpha}{2},\frac{\alpha}{2})}\right)^{(1-\beta)k}. 
 $$
\end{proof}

{\footnotesize

\providecommand{\bysame}{\leavevmode\hbox to3em{\hrulefill}\thinspace}
\providecommand{\MR}{\relax\ifhmode\unskip\space\fi MR }
\providecommand{\MRhref}[2]{%
  \href{http://www.ams.org/mathscinet-getitem?mr=#1}{#2}
}
\providecommand{\href}[2]{#2}

}

\bigskip
{\footnotesize
\hfill Uri Bader, Technion Haifa, Israel; \texttt{bader@tx.technion.ac.il}

\hfill  Pierre-Emmanuel Caprace, UCLouvain, Belgium; \texttt{pe.caprace@uclouvain.be}

\hfill  Tsachik Gelander, Hebrew University Jerusalem, Israel; \texttt{gelander@math.huji.ac.il}

\hfill  Shahar Mozes, Hebrew University Jerusalem, Israel; \texttt{mozes@math.huji.ac.il}}
\end{document}